\documentclass[preprint,12pt]{elsarticle}
\usepackage{amsthm,amsmath,amssymb}
\usepackage{graphicx}
\usepackage[colorlinks=true,citecolor=black,linkcolor=black,urlcolor=blue]{hyperref}
\usepackage{mathrsfs}
\usepackage{epstopdf}
\usepackage{epsfig}

\theoremstyle{plain}
\newtheorem{theorem}{Theorem}
\newtheorem{lemma}[theorem]{Lemma}

\theoremstyle{definition}
\newtheorem{definition}[theorem]{Definition}
\newtheorem{example}[theorem]{Example}
\newtheorem{conjecture}[theorem]{Conjecture}

\theoremstyle{remark}
\newtheorem{remark}[theorem]{Remark}

\journal{European J. Combin. }
\begin{document}
\title{Counterexamples to the interpolating conjecture on partial-dual genus polynomials of ribbon graphs}
\author{Qi Yan\\
\small School of Mathematics\\[-0.8ex]
\small China University of Mining and Technology\\[-0.8ex]
\small P. R. China\\
Xian'an Jin\footnote{Corresponding author.}\\
\small School of Mathematical Sciences\\[-0.8ex]
\small Xiamen University\\[-0.8ex]
\small P. R. China\\
\small{\tt Email:qiyan@cumt.edu.cn; xajin@xmu.edu.cn}
}

\begin{abstract}
Gross, Mansour and Tucker introduced the partial-dual orientable genus polynomial and the partial-dual Euler genus polynomial.
They showed that the partial-dual genus polynomial for an orientable ribbon graph is interpolating and gave an analogous conjecture: The partial-dual Euler-genus polynomial for any non-orientable ribbon graph is interpolating. In this paper, we first give some counterexamples to the conjecture. Then motivated by these counterexamples, we further find two infinite classes of counterexamples.
\end{abstract}


\maketitle

\section{Introduction}
We assume that the readers are familiar with the basic knowledge of topological graph theory and in particular the ribbon graphs and partial duals, and we refer the readers to \cite{BR2, CG, EM, GT}. Let $G$ be a ribbon graph and $A\subseteq E(G)$. We denote by $G^{A}$ the partial dual of $G$ with respect to $A$.

Gross, Mansour and Tucker \cite{GMT} introduced the partial-dual orientable genus polynomials for orientable ribbon graphs and the partial-dual Euler genus polynomials for arbitrary ribbon graphs.

\begin{definition}\label{def-1}\cite{GMT}
The \emph{partial-dual Euler genus polynomial} of any ribbon graph $G$ is the generating function
$$^{\partial}\varepsilon_{G}(z)=\sum_{A\subseteq E(G)}z^{\varepsilon(G^{A})}$$
that enumerates partial duals by Euler genus.
The \emph{partial-dual orientable genus polynomial} of an orientable ribbon graph $G$ is the generating function
$$^{\partial}\Gamma_{G}(z)=\sum_{A\subseteq E(G)}z^{\gamma(G^{A})}$$
that enumerates partial duals by orientable genus.
\end{definition}

They posed some research problems and made some conjectures. Conjecture 5.3 in their paper states that
\begin{conjecture}\label{con-1}\cite{GMT} (Interpolating).
The partial-dual Euler-genus polynomial $^{\partial}\varepsilon_{G}(z)$ for any non-orientable ribbon graph $G$ is interpolating.
\end{conjecture}

In this paper, we first give some counterexamples to Conjecture \ref{con-1}. Motivated by these counterexamples, we then find two infinite classes of counterexamples to the conjecture.

\section{Some counterexamples}

A {\it bouquet} is a ribbon graph having only one vertex. A \emph{signed rotation} of a bouquet is a cyclic ordering of the half-edges at the vertex and if the edge is an untwisted loop, then we give the same sign $+$ to the corresponding two half-edges, and give the different signs (one $+$, the other $-$) otherwise. The sign $+$ is always omitted. See Figure \ref{f01} for an example. Sometimes we will use the signed rotation to represent the bouquet itself. A signed rotation is called \emph{prime} if it can not be cut into two parts such that the two half-edges of each edge belong to a single part. We shall only consider prime counterexamples.

\begin{example}\label{ex-01}
Let $B$ be the bouquet with the signed rotation $$(-1, -2, 3, 4, 2, 1, 3, 4)$$ as shown in Figure \ref{f01}. We have
$^{\partial}\varepsilon_{B}(z)=4z^{2}+12z^{4}$ (see Table \ref{tab-1} for details).
\end{example}
\begin{figure}[!htbp]
\begin{center}
\includegraphics[width=9cm]{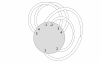}
\caption{The signed rotation of the bouquet is $(-1, -2, 3, 4, 2, 1, 3, 4)$.}
\label{f01}
\end{center}
\end{figure}
\begin{table}
\normalsize
\begin{center}
\begin{tabular}{|c|c|c|c|}
\hline
  $A$ & $\varepsilon(A)$ & $\varepsilon(A^{c})$ & $\varepsilon(B^{A})$ \\ \hline
  $\emptyset$ & 0 & 4 & 4 \\ \hline
  $\{1\}$ & 1 & 3 & 4 \\ \hline
  $\{2\}$ & 1 & 3 & 4  \\ \hline
  $\{3\}$ & 0 & 2 & 2  \\ \hline
  $\{4\}$ & 0 & 2 & 2  \\ \hline
  $\{1, 2\}$ & 2 & 2 & 4  \\ \hline
  $\{1, 3\}$ & 2 & 2 & 4  \\ \hline
  $\{1, 4\}$ & 2 & 2 & 4  \\ \hline
  $\{2, 3\}$ & 2 & 2 & 4  \\ \hline
  $\{2, 4\}$ & 2 & 2 & 4  \\ \hline
  $\{3, 4\}$ & 2 & 2 & 4  \\ \hline
  $\{1, 2, 3\}$ & 2 & 0 & 2  \\ \hline
  $\{1, 2, 4\}$ & 2 & 0 & 2  \\ \hline
  $\{1, 3, 4\}$ & 3 & 1 & 4  \\ \hline
  $\{2, 3, 4\}$ & 3 & 1 & 4  \\ \hline
  $\{1, 2, 3, 4\}$ & 4 & 0 & 4  \\ \hline
\end{tabular}
\end{center}
\caption{Euler genera of all partial duals of $(-1, -2, 3, 4, 2, 1, 3, 4)$.}
\label{tab-1}
\end{table}
Example \ref{ex-01} is the first counterexample we found, and there are no counterexamples having fewer edges.
We then found more counterexamples to Conjecture \ref{con-1} as listed in Table \ref{tab-2} with the help of computer.
\begin{table}
\normalsize
\begin{center}
\begin{tabular}{|c|c|c|}
\hline
The signed rotation of $B$ & $^{\partial}\varepsilon_{B}(z)$                            \\ \hline
$(-1, 2, 3, 4, 5, 1, 4, 5, 2, 3)$              & $8z^{2}+16z^{4}+8z^{5}$                  \\ \hline
$(-1, -2, 3, 1, 4, 2, 5, 4, 3, 5)$             & $2z+10z^{3}+8z^{4}+12z^{5}$              \\ \hline
$(-1, 2, 3, 2, 4, 5, 6, 1, 5, 6, 3, 4)$        & $8z^{2}+32z^{4}+16z^{5}+8z^{6}$          \\ \hline
$(-1, 2, 1, 3, 4, 5, 6, 2, 5, 6, 3, 4)$        & $16z^{3}+40z^{5}+8z^{6}$                 \\ \hline
$(-1, -2, 3, 1, 4, 2, 5, 4, 3, 6, 5, 6)$       & $2z+14z^{3}+12z^{4}+28z^{5}+8z^{6}$      \\ \hline
$(-1, -2, 3, 4, 5, 6, 2, 1, 5, 6, 3, 4)$       & $8z^{2}+24z^{4}+32z^{6}$                 \\ \hline
$(-1, 2, 3, 2, 4, 3, 5, 6, 7, 1, 6, 7, 4, 5)$  & $8z^{2}+48z^{4}+16z^{5}+40z^{6}+16z^{7}$ \\ \hline
$(-1, 2, 3, 4, 5, 6, 7, 1, 6, 7, 4, 5, 2, 3)$  & $16z^{2}+48z^{4}+48z^{6}+16z^{7}$        \\ \hline
$(-1, 2, 1, 3, 4, 3, 5, 6, 7, 2, 6, 7, 4, 5)$  & $16z^{3}+80z^{5}+16z^{6}+16z^{7}$        \\ \hline
$(-1, 2, 3, 4, 5, 6, 7, 1, 4, 5, 6, 7, 2, 3)$  & $32z^{4}+64z^{6}+32z^{7}$                \\ \hline
$\cdots \cdots$  & $\cdots \cdots$                \\ \hline

\end{tabular}
\end{center}
\caption{Other counterexamples to Conjecture \ref{con-1}.}
\label{tab-2}
\end{table}

\section{Two infinite classes of counterexamples}
In this section, motivated by examples in Table 2, we further give two infinite classes of counterexamples to Conjecture \ref{con-1}. The following lemma will be used.

\begin{lemma}\label{le-01}\cite{GMT}
Let $B$ be a bouquet, and let $A\subseteq E(B)$. Then $$\varepsilon(B^{A})=\varepsilon(A)+\varepsilon(A^{c}).$$
\end{lemma}

In addition, a technique called \emph{band move} in knot theory \cite{Liv} will be used, that is, a deformation of the bouquet by sliding one of the two ends of a ribbon along the boundary of the bouquet over other ribbons. This move does not change the number of boundary components. See Figures 3 and 4.

\subsection{Infinite class 1}
For each $n\geq 1$, let $B_{2n+1}$ be the bouquet with the signed rotation $$(1, 2, 3, \cdots, 2n, 2n+1, -1, 2n, 2n+1, \cdots, 2, 3),$$ as shown in Figure \ref{f02}.
\begin{figure}[!htbp]
\begin{center}
\includegraphics[width=16cm]{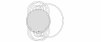}
\caption{The bouquet $B_{2n+1}$.}
\label{f02}
\end{center}
\end{figure}
Note that the edge $1$ is interlaced with all other edges and $2i, 2i+1$ are interlaced with each other for $1\leq i \leq n$.
Let $A\subseteq E(B_{2n+1})$ and $1\leq i \leq n$. If $\{2i, 2i+1\}\subseteq A$, we call $\{2i, 2i+1\}$ {\it double ribbons} of $A$. If $2i\in A$ but $2i+1\notin A$ (or $2i+1\in A$ but $2i\notin A$), we call $2i$ (or $2i+1$) a {\it single ribbon} of $A$.

\begin{lemma}\label{lem-02}
Let $A\subseteq E(B_{2n+1})$ and let $s(A)$ be the number of single ribbons of $A$. Then
\begin{eqnarray*}
\varepsilon({B_{2n+1}}^{A})=\left\{\begin{array}{ll}
                    2n+1, & \mbox{when}~s(A)=0;\\
                    2n-2s(A)+2, & \mbox{when}~s(A)>0.
                   \end{array}\right.
\end{eqnarray*}
\end{lemma}

\begin{proof}
First observe that $s(A^{c})=s(A)$.  We can assume that $1\in A$ and let $f(B)$ denote the number of boundary components of a bouquet $B$.

If the maximum labelled edge of the signed rotation of $A^{c}$ appears as double ribbons, that is, $A^{c}=(P, 2i, 2i+1, 2i, 2i+1, Q)$,
where $P$ and $Q$ are strings, then $f(A^{c})=f(P, Q)$.
If the maximum labelled edge of the signed rotation of $A^{c}$ appears as a single ribbon, that is, $A^{c}=(P, 2i, 2i, Q)$ or $A^{c}=(P, 2i+1, 2i+1, Q)$, then $f(A^{c})=f(P, Q)+1$. Repeating the previous argument leads to $f(A^{c})=s(A^{c})+1=s(A)+1$.

If the minimum labelled edge except 1 of the signed rotation of $A$ appears as a single ribbon, that is, $A=(1, 2j, P, -1, Q, 2j)$ (or $A=(1, 2j+1, P, -1, Q, 2j+1)$), then $f(A)=f(P, 2j, 1, 2j, -1, Q)$ as shown in Figure \ref{f03}.
\begin{figure}[!htbp]
\begin{center}
\includegraphics[width=12cm]{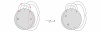}
\caption{By sliding $P$ along the outer edge of ribbon $2j$ we change $(1, 2j, P, -1, Q, 2j)$ to $(P, 2j, 1, 2j, -1, Q).$}
\label{f03}
\end{center}
\end{figure}
Since both $P$ and $Q$ do not contain edge $1$ and $s(P, Q)=s(A)-1$, it follows that
$$f(A)=f(P, Q)=s(P, Q)+1=s(A).$$
If the minimum labelled edge except 1 of the signed rotation of $A$ appears as double ribbons, that is, $A=(1, 2j, 2j+1,P, -1, Q, 2j, 2j+1)$,
then $$f(A)=f(2j, 2j+1, 2j, 2j+1, 1, P, -1, Q)$$ as shown in Figure \ref{f04}. If $s(A)=0$, then
\begin{eqnarray*}
f(A)&=&f(2j, 2j+1, 2j, 2j+1, 1, P, -1, Q)\nonumber\\
&=&f(1, P, -1, Q)=\cdots=f(1,-1)=1.
\end{eqnarray*}
Otherwise, $s(A)>0$, then repeat the above process, we have $f(A)=s(A).$ Therefore,
\begin{eqnarray*}
f(A)=\left\{\begin{array}{ll}
                    1, & \mbox{when}~s(A)=0;\\
                    s(A), & \mbox{when}~s(A)>0.
                   \end{array}\right.
\end{eqnarray*}

\begin{figure}[!htbp]
\begin{center}
\includegraphics[width=12cm]{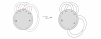}
\caption{By sliding the ribbon 1 along the boundary of the handle formed by ribbons $2j$ and $2j + 1$ we change $(1, 2j, 2j+1,P, -1, Q, 2j, 2j+1)$ to $(2j, 2j+1, 2j, 2j+1, 1, P, -1, Q).$}
\label{f04}
\end{center}
\end{figure}

Since $\varepsilon({B_{2n+1}}^{A})=\varepsilon(A)+\varepsilon(A^{c})$ by Lemma \ref{le-01} and $$1-|A|+f(A)=2-\varepsilon(A), 1-|A^{c}|+f(A^{c})=2-\varepsilon(A^{c})$$ by the Euler formulas, it follows that
\begin{eqnarray*}
\varepsilon({B_{2n+1}}^{A})&=&|A|+|A^{c}|-f(A)-f(A^{c})+2\nonumber\\
                           &=&\left\{\begin{array}{ll}
                    2n+1, & \mbox{when}~s(A)=0;\\
                    2n-2s(A)+2, & \mbox{when}~s(A)>0.
                   \end{array}\right.
\end{eqnarray*}

\end{proof}

\begin{theorem}\label{the-01}
The partial-dual Euler genus polynomial for $B_{2n+1}$ is given
by the formula $$^{\partial}\varepsilon_{B_{2n+1}}(z)=2^{n+1}z^{2n+1}+\sum_{s = 1}^n 2^{n+1}\dbinom{n}{s} z^{2n-2s+2}.$$
\end{theorem}
\begin{proof}
By Lemma \ref{lem-02}, we have the following two cases:
\begin{enumerate}
  \item $\varepsilon({B_{2n+1}}^{A})=2n+1$ if and only if $s(A)=0$. Then we can choose $A$ in $2^{n+1}$ ways.
  \item $\varepsilon({B_{2n+1}}^{A})=2n-2s+2$ where $1\leq s \leq n$ if and only if $s(A)=s$. Then we can choose $s$ single edges for the ribbon subset $A$ in $\dbinom{n}{s}2^{s}$ ways and then select the remaining double edges in $2^{n-s}$ ways, and $A$ may or may not contain the edge $1$.
      Hence, we can choose $A$ in $2^{n+1} \dbinom{n}{s}$ ways.
\end{enumerate}

We obtain the formula.
\end{proof}

\begin{remark}
By Theorem \ref{the-01}, $B_{2n+1}$ is a counterexample for each $n\geq 2$.
\end{remark}

\subsection{Infinite class 2}

For each $n\geq 1$, let $C_{2n+2}$ be the bouquet with the signed rotation $$(1, 2, 3, 4, \cdots, 2n+1, 2n+2, -2, -1, 2n+1, 2n+2, \cdots, 3, 4)$$ as shown in Figure \ref{f05}.
\begin{figure}[!htbp]
\begin{center}
\includegraphics[width=12cm]{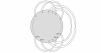}
\caption{The bouquet $C_{2n+2}$.}
\label{f05}
\end{center}
\end{figure}
Note that the edges $1, 2$ are interlaced with all other edges but $1, 2$ are not interlaced with each other and $2i+1, 2i+2$ are interlaced with each other for $1\leq i \leq n$. Let $A\subseteq E(C_{2n+2})$ and $1\leq i \leq n$. If $\{2i+1, 2i+2\}\subseteq A$, we call $\{2i+1, 2i+2\}$ {\it double ribbons} of $A$. If $2i+1\in A$ but $2i+2\notin A$ (or $2i+2\in A$ but $2i+1\notin A$), we call $2i+1$ (or $2i+2$) a {\it single ribbon} of $A$.

\begin{lemma}\label{lem-03}
Let $A\subseteq E(C_{2n+2})$ and let $s(A)$ be the number of single ribbons of $A$.
Then $\varepsilon({C_{2n+2}}^{A})=$
\begin{eqnarray*}
\left\{\begin{array}{ll}
                    2n+2, & \mbox{if one of 1, 2 is in}~A~\mbox{and the other is in}~A^{c}~\mbox{and}~s(A)=0;\\
                    2n-2s(A)+4, & \mbox{if one of 1, 2 is in}~A~\mbox{and the other is in}~A^{c}~\mbox{and}~s(A)>0;\\
                    2n-2s(A)+2, & \mbox{if 1, 2 are both in}~A~\mbox{or both in}~A^{c}.
                   \end{array}\right.
\end{eqnarray*}
\end{lemma}

\begin{proof}
We know that $s(A)=s(A^{c})$.  If one of 1, 2 is in $A$ and the other is in $A^{c}$, then, as
in the proof of Lemma \ref{lem-02}, we have
\begin{eqnarray*}
f(A)=f(A^{c})=\left\{\begin{array}{ll}
                    1, & \mbox{when}~s(A)=0;\\
                    s(A), & \mbox{when}~s(A)>0.
                   \end{array}\right.
\end{eqnarray*}

If $1, 2\in A$, then $f(A^{c})=s(A)+1$.

\begin{enumerate}
  \item If the minimum labelled edge except 1, 2 of the signed rotation of $A$ appears as a single ribbon, that is,
$A=(1, 2, 2j+1, P, -2, -1, Q, 2j+1)$ (or $A=(1, 2, 2j+2, P, -2, -1, Q, 2j+2)$), then $$f(A)=f(P, 2j+1, 1, 2, 2j+1, -2, -1, Q).$$
Since both $P$ and $Q$ do not contain edges $1, 2$ and $s(P, Q)=s(A)-1$, it follows that
\begin{eqnarray*}
f(A)&=&f(P,Q)+f(2j+1, 1, 2, 2j+1, -2, -1)-1\nonumber\\
&=&f(P, Q)+1=(s(P, Q)+1)+1=s(A)+1.
\end{eqnarray*}
  \item If the minimum labelled edge except 1, 2 of the signed rotation of $A$ appears as double ribbons, that is, $$A=(1, 2, 2j+1, 2j+2,P, -2, -1, Q, 2j+1, 2j+2),$$
then $$f(A)=f(2j+1, 2j+2, 2j+1, 2j+2, 1, 2, P, -2, -1, Q).$$
If $s(A)=0$, then
\begin{eqnarray*}
f(A)&=&f(2j+1, 2j+2, 2j+1, 2j+2, 1, 2, P, -2, -1, Q)\nonumber\\
&=&f(1, 2, P, -2, -1, Q)=\cdots=f(1, 2, -2, -1)=1.
\end{eqnarray*}
Otherwise, $s(A)>0$, then repeat the above process, we have $f(A)=s(A)+1.$ Therefore,
we can see that $f(A)=s(A)+1$ for $0\leq s(A)\leq n.$
\end{enumerate}

Since $\varepsilon({C_{2n+2}}^{A})=\varepsilon(A)+\varepsilon(A^{c})$ by Lemma \ref{le-01} and $$1-|A|+f(A)=2-\varepsilon(A), 1-|A^{c}|+f(A^{c})=2-\varepsilon(A^{c}),$$ it follows that
$\varepsilon({C_{2n+2}}^{A})=|A|+|A^{c}|-f(A)-f(A^{c})+2=$
\begin{eqnarray*}
\left\{\begin{array}{ll}
                    2n+2, & \mbox{if one of 1, 2 is in}~A~\mbox{and the other is in}~A^{c}~\mbox{and}~s(A)=0;\\
                    2n-2s(A)+4, & \mbox{if one of 1, 2 is in}~A~\mbox{and the other is in}~A^{c}~\mbox{and}~s(A)>0;\\
                    2n-2s(A)+2, & \mbox{if 1, 2 are both in}~A~\mbox{or both in}~A^{c}.
                   \end{array}\right.
\end{eqnarray*}
\end{proof}

\begin{theorem}\label{the-02}
The partial-dual Euler genus polynomial for $C_{2n+2}$ is given
by the formula $$^{\partial}\varepsilon_{C_{2n+2}}(z)=2^{n+1}(n+2)z^{2n+2}+\sum_{s = 2}^n 2^{n+1} \left(\dbinom{n}{s}+\dbinom{n}{s-1}\right)z^{2n-2s+4}+2^{n+1}z^{2}.$$
\end{theorem}
\begin{proof}
By Lemma \ref{lem-03}, we have the following three cases:
\begin{enumerate}
  \item $\varepsilon({C_{2n+2}}^{A})=2n+2$ if and only if one of $1, 2$ is in $A$ and the other is in $A^{c}$ and $A$ has no single ribbons (or only one single ribbon) or $1, 2$ are both in $A$ or both in $A^{c}$ and $A$ has no single ribbons. Then we can choose $A$ in $2^{n+1}+2^{n+1}\dbinom{n}{1}+2^{n+1}=2^{n+1}(n+2)$ ways.
  \item $\varepsilon({C_{2n+2}}^{A})=2n-2s+4$ where $2\leq s \leq n$ if and only if one of $1, 2$ is in $A$ and the other is in $A^{c}$ and $A$ has $s$ single ribbons or $1, 2$ are both in $A$ or both in $A^{c}$ and $A$ has $s-1$ single ribbons. Then we can choose $A$ in
      $2^{n+1} \left(\dbinom{n}{s}+\dbinom{n}{s-1}\right)$ ways.
  \item $\varepsilon({C_{2n+2}}^{A})=2$ if and only if $1, 2$ are both in $A$ or both in $A^{c}$ and $A$ has $n$ single ribbons. Then we can choose $A$ in $2^{n+1}$ ways.
\end{enumerate}
\end{proof}

\begin{remark}
By Theorem \ref{the-02}, $C_{2n+2}$ is a counterexample for each $n\geq 1$. In particular, when $n=1$, $C_{4}$ is exactly the bouquet in Example \ref{ex-01}.
\end{remark}

\section{Acknowledgements}
This work is supported by NSFC (Nos. 12171402, 12101600) and the Fundamental Research Funds for the Central Universities  (Nos. 20720190062, 2021QN1037). We thank the referees sincerely for their valuable comments.

%\section*{References}
\bibliographystyle{model1b-num-names}
\bibliography{<your-bib-database>}
%\bibliographystyle{amcjoucc}
%\bibliography{amcexample}

\end{document}